\documentclass[letterpaper, 10 pt, conference]{ieeeconf}  

\IEEEoverridecommandlockouts                              

\usepackage{graphicx}
\usepackage{caption}
\graphicspath{
    {figures/}
}

\usepackage{amsmath,amssymb,enumerate}

\usepackage{amsthm}
\usepackage{setspace}
\usepackage{booktabs}
\usepackage[usenames,dvipsnames,svgnames,table]{xcolor}
\usepackage{mathtools}
\usepackage{algorithm, algorithmicx, algpseudocode}

\usepackage{blindtext}
\usepackage{gensymb}
\usepackage{xparse}
\usepackage{lipsum}
\usepackage{mathrsfs}
\usepackage[mathscr]{euscript}
\usepackage{times}
\usepackage{cite} 
\usepackage{multicol}
\definecolor{darkgreen}{rgb}{0,0.6,0}
\usepackage[bookmarks=true,colorlinks=true,pdfpagemode=UseNone,citecolor=darkgreen,linkcolor=black,urlcolor=BrickRed,pagebackref]{hyperref}
\usepackage[caption=false,font=footnotesize]{subfig}
\usepackage{amsfonts}
\usepackage[utf8]{inputenc}
\usepackage[T1]{fontenc}
\usepackage{textcomp}
\usepackage{arydshln}
\usepackage{balance}

\newtheorem{problem}{Problem}
\newtheorem{theorem}{Theorem}

\newtheorem{proposition}[theorem]{Proposition}

\newtheorem{remark}{Remark}

\newtheorem{definition}{Definition}

\renewcommand{\thefigure}{\arabic{figure}}

\definecolor{note}{rgb}{0.1,0.1,1}
\definecolor{rephase}{rgb}{0.15,0.7,0.15}
\definecolor{bag}{rgb}{0.6,0.6,0.2}

\makeatletter
\renewcommand*\env@matrix[1][c]{\hskip -\arraycolsep
  \let\@ifnextchar\new@ifnextchar
  \array{*\c@MaxMatrixCols #1}}
\makeatother

\newcommand{\transpose}{\mathsf{T}}

\makeatletter
\newcommand{\mathleft}{\@fleqntrue\@mathmargin0pt}
\newcommand{\mathcenter}{\@fleqnfalse}
\makeatother

\definecolor{orange}{RGB}{255,127,0}

\title{\LARGE \bf Lie Algebraic Cost Function Design for Control on Lie Groups} 

\author{Sangli Teng, William Clark, Anthony Bloch, Ram Vasudevan, Maani Ghaffari
\thanks{Toyota Research Institute provided funds to support this work. Funding for M. Ghaffari was in part provided by NSF Award No. 2118818. W. Clark was supported by NSF grant DMS-1645643. A. Bloch was supported in part by NSF grant DMS-2103026 and AFOSR grant FA0550-18-0028.}
\thanks{S. Teng, A. Bloch, R. Vasudevan, and M. Ghaffari are with the University of Michigan, Ann Arbor, MI 48109, USA. \texttt{\{sanglit,abloch,ramv,maanigj\}@umich.edu}.}%
\thanks{W.~Clark is with the department of Mathematics, Cornell University, Ithaca, NY. {\tt\small wac76@cornell.edu}}
}

\begin{document}

\maketitle
\thispagestyle{empty}
\pagestyle{empty}

\begin{abstract}
This paper presents a control framework on Lie groups by designing the control objective in its Lie algebra. Control on Lie groups is challenging due to its nonlinear nature and difficulties in system parameterization. Existing methods to design the control objective on a Lie group and then derive the gradient for controller design are non-trivial and can result in slow convergence in tracking control. We show that with a proper left-invariant metric, setting the gradient of the cost function as the tracking error in the Lie algebra leads to a quadratic Lyapunov function that enables globally exponential convergence. In the PD control case, we show that our controller can maintain an exponential convergence rate even when the initial error is approaching $\pi$ in SO(3). We also show the merit of this proposed framework in trajectory optimization. The proposed cost function enables the iterative Linear Quadratic Regulator (iLQR) to converge much faster than the Differential Dynamic Programming (DDP) with a well-adopted cost function when the initial trajectory is poorly initialized on SO(3). 


\end{abstract} 

\IEEEpeerreviewmaketitle

\section{Introduction} 

Geometric control techniques that incorporate differential geometry \cite{bullo2019geometric} with control theory have been applied to many robotics systems, e.g., legged robots \cite{yanran-NMPC-variational, teng2022error, Chig-VBL, vblMPC} and unmanned aerial vehicles (UAV) \cite{sreenath2013geometric, lee2010geometric}. For systems on Lie groups, geometric thinking enables a better choice of coordinates. Therefore, issues of local coordinates, such as singularities in Euler angles \cite{Euler-survey} and poor linearization in observer design \cite{barrau2017invariant, huang2010observability} can be avoided. Despite the merits, describing systems on the Lie group also introduces difficulties in cost function design and the analysis of its derivatives.

A Proportional-Derivative (PD) controller has been proposed and applied to control fully actuated mechanical systems \cite{bullo1999tracking} by defining the configuration and velocity error on a Riemannian manifold. This framework has also been applied to Lie groups such as SE(3) for UAV control \cite{lee2010geometric, lee2011geometric}. The trace function \cite{koditschek1989application} has been introduced and applied in \cite{bullo1999tracking, lee2010geometric, vblMPC, Wu-VBL, Chig-VBL, yanran-NMPC-variational} to indicate the configuration error of rotational motion. However, this error function may lead to slow error convergence \cite{johnson2022globally, teng2022error} when the rotational error is large. To solve the above problem, the logarithmic error has been applied \cite{yu2015high, johnson2022globally, teng2022error}. However, the \cite{yu2015high, teng2022error} does not prove the stability property. The proof in \cite{johnson2022globally} is specific to SO(3) and requires the left Jacobian of SO(3), thus making it less general for systems on Lie groups. 



\begin{figure}[t]
\label{fig:cover}
    \centering
    \includegraphics[width=\columnwidth]{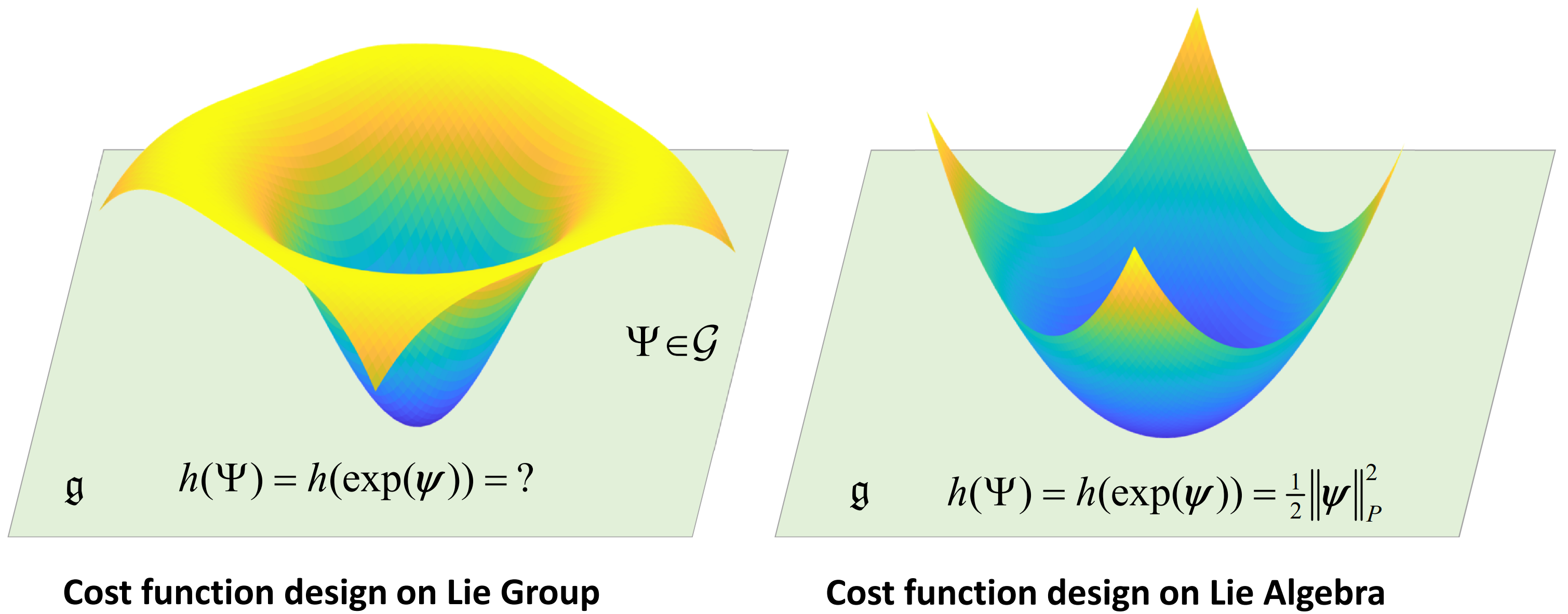}
    \caption{We derive the cost function design for control problems on Lie groups. However, conventional methods design the cost function on the group, which may introduce undesired behavior such as slow convergence. For the SO(3) case, using the trace function to indicate configuration error results in a cost function on the left, whose gradient vanishes when the error becomes large. Instead, we design a quadratic cost function in the Lie algebra that ensures exponential stability and accelerates trajectory optimization.}
    \label{fig:cover}
\end{figure}

Optimization-based control using geometric methods has also been explored in recent years. Research in  \cite{watterson2020trajectory, bonalli2019trajectory} applied optimization methods to generate optimal trajectories on Riemannian manifolds. A factor graph-based optimization-based control is proposed in \cite{ta2014factor} that derived the gradient on manifolds for optimization. The work of \cite{ct-LQR-LieGroup} proposed the Lie group projection operator Newton method for continuous-time control on Lie groups. A discrete-time Differential Dynamic Programming (DDP) algorithm is proposed on Lie groups in \cite{lie-ddp}. The main procedure of \cite{ct-LQR-LieGroup, lie-ddp} is to derive the local perturbed system and then solve a local optimal control problem via dynamic programming. Superlinear convergence is possible for both cases. The main drawback of these works is the derivation of the cost function and linearization. They mainly apply to general Riemannian geometries while not fully utilizing the symmetry of Lie groups. 

In this work, we focus on control problems on Lie groups. We exploit the existing symmetry structure in Lie groups and the fact that the Lie algebra of a Lie group completely captures the local structure of the group. This approach enables one to represent the system in a vector space and measure the distance between two arbitrary configurations. Moreover, designing the cost function in the Lie algebra enables a more concise formulation for all connected matrix Lie groups. Though this idea is intuitive, it is not trivial as the equation of motion is nonlinear and represented on the group. Therefore, bridging the gap between the cost function in the Lie algebra and the equation of motion on the Lie group is the key novelty of this work. Fig.~\ref{fig:cover} illustrates the proposed idea. 

In particular, this paper has the following contributions:
\begin{enumerate}
    \item a control framework on Lie groups with the gradient of the cost function described in the corresponding Lie algebra; 
    \item a proof illustrating that a quadratic Lyapunov function of the configuration error in the Lie algebra can enable globally exponential stability; 
    \item a PD tracking controller and an iterative Linear Quadratic Regulator method using the proposed control framework, and
    \item a numerical simulation that illustrates that the proposed method exhibits better performance than existing methods even when the initialization is chosen poorly.
\end{enumerate}
The remaining of this paper is organized as follows. Section~\ref{sec:prelim} presents the necessary mathematical and control backgrounds. Section~\ref{sec:cost} introduces the main results of designing the cost function in the Lie algebra. 
A PD tracking controller and an iLQR-based trajectory optimization method on SO(3) are introduced in Section~\ref{sec:design}. Section~\ref{sec:sim} gives the results of the numerical simulation. Discussion is presented in Section~\ref{sec:discuss}, and the conclusion is presented in Section~\ref{sec:conclusion}.


\section{Preliminary}
\label{sec:prelim}
This section provides an overview of the mathematical  backgrounds regarding the proposed approach. 

\subsection{Differential geometry}

Let $M$ denote a smooth manifold. The tangent (cotangent) space of $M$ at $x \in M$ is denoted by $T_xM$ ($T_x^{*}M$). We equip the manifold $M$ with a Riemannian metric, i.e., a smooth map associating to each tangent space $T_x M$ an inner product $\langle \cdot, \cdot \rangle_x$. 
Given a real-valued function $f$ on manifold $M$, the gradient is the vector field $\nabla f$ such that
$$\langle \nabla f, Y \rangle:=\mathcal{L}_{Y}f,$$
where $\mathcal{L}_{Y}f$ is the Lie derivative of $f$ with respect to the smooth vector field $Y$.

\subsection{Lie group}
Let $\mathcal{G}$ be an $n$-dimensional matrix Lie group and $\mathfrak{g}$ its associated Lie algebra (hence, $\dim \mathfrak{g} = n$)~\cite{chirikjian2011stochastic,hall2015lie}. For convenience, we define the following isomorphism
\begin{equation}
    (\cdot)^\wedge:\mathbb{R}^n \rightarrow \mathfrak{g},
\end{equation}
that maps an element in the vector space $\mathbb{R}^n$ to the tangent space of the matrix Lie group at the identity. We also define the inverse of $(\cdot)^\wedge$ map as:
\begin{equation}
    (\cdot)^\vee:\mathfrak{g} \rightarrow \mathbb{R}^n. 
\end{equation}
Then, for any $\phi \in \mathbb{R}^{n}$, we can define the Lie exponential map as
\begin{equation}
    \exp(\cdot):\mathbb{R}^{n} \rightarrow \mathcal{G},\ \ \exp(\phi)=\operatorname{exp_m}({\phi}^\wedge),
\end{equation}
where $\operatorname{exp_m}(\cdot)$ is the exponential of square matrices. We also define the Lie logarithmic map as the inverse of Lie exponential map
\begin{equation}
    \log(\cdot): \mathcal{G} \rightarrow \mathbb{R}^{n}. \ \ 
\end{equation}
For every $X \in \mathcal{G}$, the adjoint action, $\mathrm{Ad}_{X}: \mathfrak{g}\rightarrow \mathfrak{g}$, is a Lie algebra isomorphism that enables change of frames 
\begin{equation}
    \mathrm{Ad}_{X}({\phi}^\wedge)= X{{\phi}^\wedge}X^{-1}.
\end{equation}
Its derivative at the identity gives rise to the adjoint map in the Lie algebra as
\begin{equation}
    \mathrm{ad}_{\phi}\eta = [{\phi}^\wedge, {\eta}^\wedge],
\end{equation}
where $\phi^\wedge, \eta^\wedge \in \mathfrak{g}$ and $[\cdot, \cdot]$ is the Lie Bracket.
For a trajectory on a Lie group, we have the reconstruction equation
\begin{equation}
    \frac{d}{dt}X_t = X_t \xi_t^{\wedge},
\end{equation}
where $X_t \in \mathcal{G}$ and $\xi_t^\wedge \in \mathfrak{g}$.

\subsection{Configuration error dynamics}
Consider the trajectory $X_t$ and the nominal trajectory $X_{d,t}$ on $\mathcal{G}$ and the corresponding twist $\xi_t$ and $\xi_{d,t}$, respectively. Then
\begin{align*}
    \frac{d}{dt} X_t = X_t {\xi}^\wedge_t, \text{ and } \frac{d}{dt} X_{d, t} = X_{d,t} {\xi}^\wedge_{d,t}.
\end{align*}
Similar to the left or right error defined in \cite{bullo1999tracking}, we define the error between $X_t^d$ and $X_t$ as
\begin{equation}
\label{eq:X_err}
    \Psi_t = X_{d,t}^{-1} X_t \in \mathcal{G}. 
\end{equation}
For the tracking problem, our goal is to drive the error from the initial condition $\Psi_0$ to the identity $I \in \mathcal{G}$. Taking derivatives on both sides of~\eqref{eq:X_err}, we have
\begin{align*}
    \dot{\Psi}_t &=X_{d,t}^{-1}\frac{d}{dt} X_t + \frac{d}{dt} (X_{d,t}^{-1}) X_t \\
    &=X_{d,t}^{-1}\frac{d}{dt} X_t-X_{d,t}^{-1}\frac{d}{dt}(X_{d,t})X_{d,t}^{-1} X_t\\
    &=X_{d,t}^{-1}X_t{\xi}^\wedge_t - X_{d,t}^{-1}X_{d,t} {\xi}^\wedge_{d,t}X_{d,t}^{-1}X_t 
    =\Psi_t{\xi}^\wedge_t-{\xi}^\wedge_{d,t}\Psi_t.
\end{align*}
Therefore, 
\begin{equation}
\label{eq:error_dynamics}
\begin{aligned}
\dot{\Psi}_t &= \Psi_t({\xi}_t - \Psi_t^{-1}{\xi}_{d,t}\Psi_t)^\wedge =\Psi_t({\xi}_t - \mathrm{Ad}_{\Psi_t^{-1}}{\xi}_{d,t})^\wedge.
\end{aligned}
\end{equation}
To indicate the difference between two configurations on $\mathcal{G}$, we introduce an error function $h$ defined on $\mathcal{G}$. 
\begin{definition}
\label{def:error_fun}
Given the configuration error $\Psi$, we say that a function $h$ is an error function if it is positive definite, such that $h(\Psi) \ge 0$ for any $\Psi$, and $h(\Psi) = 0$ if and only if $\Psi = I$. We also say that the error function is symmetric if $h(\Psi) = h(\Psi^{-1})$.
\end{definition}

\subsection{Lyapunov stability theory}
We introduce a Lyapunov stability theorem that can certify the stability of a dynamics system.  
\begin{definition}
Let $V: \mathcal{D} \to \mathbb{R}$ be a continuously differentiable function, such that 
\begin{equation}
    V(0) = 0\ \ \text{and} \ \ V(x) \ge 0\ \  \text{for all}\ \ x \in \mathcal{D}\setminus0,
\end{equation}
\begin{equation}
    \dot{V}(x) \le 0\ \ \text{for all}\ \ x \in \mathcal{D}.
\end{equation}
Then we say $V$ is a Lyapunov function \cite{khalil2002nonlinear}. 
\end{definition}
A stronger condition is the exponential stability. 
\begin{definition}
\label{def:ES_lyap}
Let $V$ be a Lyapunov function and suppose there exist constants $c_1, c_2, c_3 > 0$, such that
\begin{equation}
\label{eq:ES_stable}
\begin{aligned}
c_1 \|x\|^2 \le V(x) \le c_2 \|x\|^2, \\
\dot{V}(x) \le -c_3\|x\|^2. 
\end{aligned}
\end{equation}
Then the origin is exponentially stable, i.e., there exist \mbox{$\alpha, \beta > 0$} such that 
\begin{equation}
    \|x(t)\| \le \beta \|x(0)\|\exp{(-\alpha t)}.
\end{equation}
If \eqref{eq:ES_stable} holds globally, then the origin is globally stable \cite{khalil2002nonlinear}.
\end{definition}


\section{Cost function design}
\label{sec:cost}
A wide range of the geometric control literature considers the design of the error or cost function on manifolds~\cite{bullo1999tracking, ct-LQR-LieGroup, lie-ddp}. These methods are general, but they introduce difficulties while deriving the gradient or Hessian matrix of a cost function on the manifold. Additionally, they do not fully utilize the symmetry of Lie groups, especially the properties of the Lie algebra. 

To address this challenge, we instead design the gradient of the error function and show that the corresponding error function can satisfy certain stability properties. 


\subsection{Cost function design via gradient}
Here we start with a tracking control problem of regulating the error $\Psi_t$ between two trajectories. Let \mbox{$h:=h(\Psi)$} denotes the candidate error function. 
Its time derivative is 
\begin{equation}
\begin{aligned}
\frac{d}{dt}h &= \mathcal{L}_{\dot{\Psi}}h \\
&= \langle \nabla h, \dot{\Psi} \rangle \\
&= \langle \nabla h, \Psi ({\xi} - \mathrm{Ad}_{\Psi^{-1}}{\xi}_{t})^\wedge \rangle . 
\end{aligned}
\end{equation}
Now we can define a metric and shape the convergence property. For the tracking control, we define the following left-invariant inner product on $\mathcal{G}$.
\begin{definition}
\label{def:inner_product}
Given $\phi_1, \phi_2\in \mathbb{R}^{\dim \mathfrak{g}}$ and $\phi_1^{\wedge}, \phi_2^{\wedge} \in \mathfrak{g}$, we define the inner product $\langle \phi_1^{\wedge}, \phi_2^{\wedge}\rangle_{\mathfrak{g}} = \phi_1^{\transpose} P \phi_2 $,
where $P$ is a positive definite matrix. This inner product is left-invariant. To see this, suppose $X\phi_1^\wedge, X\phi_2^\wedge \in T_{X}\mathcal{G}$, $ \forall{X} \in \mathcal{G}$, then 
\begin{align*}
    \langle X\phi_1^{\wedge}, X\phi_2^{\wedge}\rangle_X &= \langle (\ell_{X^{-1}})_* X\phi_1^{\wedge}, (\ell_{X^{-1}})_* X\phi_2^{\wedge}\rangle_{\mathfrak{g}} \\
    &= \langle \phi_1^{\wedge}, \phi_2^{\wedge}\rangle_{\mathfrak{g}}.
\end{align*}
Note that $(\ell_{X^{-1}})_* = X^{-1}:T_{X}\mathcal{G} \to \mathfrak{g}$ is the pushforward map.

\begin{definition}
We denote the norm induced by inner product in \eqref{def:inner_product} as $\|\phi\|_{P}$, $\phi^{\wedge} \in \mathfrak{g}$, such that
\begin{equation}
    \| \phi \|_{P} = \sqrt{\langle \phi^{\wedge}, \phi^{\wedge} \rangle}_{\mathfrak{g}} = \sqrt{\phi^{\transpose}P \phi}.
\end{equation}
\end{definition}
\begin{remark}
In controller design or verification of stability, we do not need to design $P$. We only need to verify that $P$ exists and is positive definite. 
\end{remark}
\end{definition}

With the inner product in Definition \ref{def:inner_product}, we now design the gradient of the candidate error function $h$. By the exponential map, we have the error in Lie algebra $\psi$, such that 
$$\Psi=\exp(\psi),\ \Psi\in\mathcal{G},\ {\psi}^{\wedge}\in\mathfrak{g}.$$
As we wish the equilibrium of the tracking control to be exponentially stable,  we mimic the quadratic Lyapunov function in linear control and design $\nabla h$ to be linear to $\psi$,
\begin{equation}
    \nabla h = \Psi\psi^{\wedge}. 
\end{equation}
Thus, a feedback for the first order system \eqref{eq:error_dynamics} is
\begin{equation}
    \label{eq:lin_fb_1st}
    \xi = -K\psi + \mathrm{Ad}_{\Psi^{-1}}\xi_d. 
\end{equation}
We then have the corresponding closed loop system as follows.
\begin{equation}
\label{eq:dhdt}
\begin{aligned}
\frac{d}{dt}h & = \langle\Psi\psi^{\wedge}, -\Psi (K\psi)^{\wedge} \rangle_{\Psi} \\
& = \langle\psi^{\wedge}, - (K\psi)^{\wedge} \rangle_{\mathfrak{g}} \\
& = -\psi^{\transpose}\frac{( PK + K^{\transpose}P )}{2}\psi \\
& =: -\psi^{\transpose}Q\psi
\end{aligned}
\end{equation}
By the Lyapunov equation
\begin{equation}
\label{eq:lyap_eq}
    P(-K) + (-K)^{\transpose}P + 2Q = 0, 
\end{equation}
we find that any gain matrix $K$ with only positive eigenvalues will ensure that for any $Q>0$, there is a positive definite matrix $P$ that satisfies \eqref{eq:lyap_eq}. 

\subsection{Existence and scaling of the candidate function }
We have shaped the time derivative of the candidate error function $h$. The remaining issue is to show the existence and the scaling of $h$. Now we prove that $h$ exists and it is a quadratic function of $\psi$. 


Let $\mathrm{d}$ denote the differential. We can write $\mathrm{d}\Psi \in T_{\Psi}^{*}\mathcal{G}$ and $\mathrm{d}\Psi = \Psi \mathrm{d}\eta^{\wedge}$, where $\mathrm{d}\eta^{\wedge} \in \mathfrak{g}^*$ and $\eta^{\wedge} \in \mathfrak{g}$. Then
\begin{equation}
\begin{aligned}
    \mathrm{d}{h} &= \langle\nabla h,  \mathrm{d}\Psi \rangle_{\Psi} = \langle \Psi \psi^{\wedge},  \Psi \mathrm{d}\eta^{\wedge} \rangle_{\Psi} \\
    &= \langle \psi^{\wedge},  \mathrm{d}\eta^{\wedge} \rangle_{\mathfrak{g}} = \psi^{\transpose}P\mathrm{d}\eta \\
    &= \mathrm{d}(\frac{1}{2}\psi^{\transpose}P\psi).
\end{aligned}
\end{equation}


Therefore, we have $$h = \frac{1}{2}\psi^\transpose P \psi + C, \quad  C\in \mathbb{R}.$$ We let $C = 0$ and we show that $h$ is a quadratic function of $\psi$ that satisfy Definition \ref{def:error_fun}.

Now we show that the linear feedback \eqref{eq:lin_fb_1st} can exponentially stabilize the equilibrium. By the Rayleigh quotient argument \cite{horn_johnson_1985} we can show that there exist constants $a_1, a_2 > 0$, such that
$$ a_1 \|\psi\|^{2} \le h = \frac{1}{2}\|\psi\|^2_{P} \le a_2 \|\psi\|^2.$$
There also exist constants $b_1, b_2 > 0$ such that 
$$b_1 \|\psi\|^2 \le -\dot{h} = \|\psi\|_{Q}^2 \le b_2 \|\psi\|^2 .$$
Finally we have
$$ \dot{h} \le -b_1 \|\psi\|^2 \le -\frac{b_1}{a_2}h.$$
Thus, by Definition \ref{def:ES_lyap}, we show that the equilibrium $\psi = 0$, $\Psi = \exp{(0)} = I$ is exponentially stable. Additionally, as this condition holds for any $\psi$, the equilibrium $\psi = 0$ is also globally stable. 

\subsection{Main theorem}
We now present the main theorem by showing that $$h:=h(\exp{(\psi)}) = \frac{1}{2}\|\psi\|^2_P$$
is a Lyapunov candidate function for equilibrium $\psi = 0$ and provide its gradient. 
\begin{theorem} 
\label{theorem:Lyap_candiadate}
Consider the state $X \in \mathcal{G}$, $\phi \in \mathbb{R}^{\dim \mathfrak{g}}$, and $X = \exp{(\phi)}$. We consider the metric in Definition \ref{def:inner_product}. The function $h = \frac{1}{2} \|\phi\|^2_{P}$ is a candidate Lyapunov function and the gradient of $h$ with respect to $X$ is
\begin{equation}
    \nabla h = X\phi^{\wedge}.
\end{equation}
\end{theorem}

\begin{proof}
Let $h = \frac{1}{2}\|\phi\|_{P}^2$, $\phi \in \mathfrak{g}$. $h$ is positive definite and $h=0$ if and only if $\phi = 0$. Then we compute the differential of $h$ as
\begin{equation}
    \begin{aligned}
        \mathrm{d}h & = \phi^{\transpose}P \mathrm{d}\eta = \langle \phi^{\wedge}, \mathrm{d}\eta^{\wedge} \rangle_{\mathfrak{g}} \\
        & = \langle X\phi^{\wedge}, X\mathrm{d}\eta^{\wedge} \rangle_{X}\\
        & = \langle X\phi^{\wedge}, \mathrm{d}X \rangle_{X} .
    \end{aligned}
\end{equation}
Thus, we have the gradient
$$\nabla h = X\phi^{\wedge}.$$
\end{proof}

\begin{proposition}
Consider the state in Theorem~\ref{theorem:Lyap_candiadate}, let $\phi^{\wedge}, \phi^{*\wedge} \in \mathfrak{g}$. $h = \frac{1}{2}\|\phi - \phi^{*}\|_{P}^2$ is a candidate Lyapunov function for equilibrium $\phi = \phi^{*}$. The gradient of $h$ with respect to $h$ is
\begin{equation}
    \nabla h = X(\phi - \phi^{*})^{\wedge}.
\end{equation}
\end{proposition}
\begin{proof}
Similar to the proof of Theorem~\ref{theorem:Lyap_candiadate}. 
\end{proof}

\begin{theorem}
Consider the state in Theorem~\ref{theorem:Lyap_candiadate} as a trajectory. Let $\xi^\wedge \in \mathfrak{g}$. The system $$\dot{X} = X\xi^{\wedge}$$ can be exponentially stabilized to $X = I$ by linear feedback $\xi = - K\phi$, where $K$ is a gain matrix with only positive eigenvalues. 
\end{theorem}
\begin{proof}
The condition can be verified by the time derivative of the Lyapunov function defined in Theorem \ref{theorem:Lyap_candiadate} following the same process in \eqref{eq:dhdt}. The existence of such a metric \eqref{def:inner_product} can be justified by the Lyapunov equation \eqref{eq:lyap_eq}. 
\end{proof}

\section{Controller design on SO(3)}
\label{sec:design}
In this section, we propose a PD controller and an iLQR trajectory optimization algorithm given the candidate Lyapunov function we proposed in the last section. Both frameworks are implemented in SO(3). 
\subsection{Dynamics on SO(3)}
Now consider the rotational motion of a 3D rigid body. The state of the robot can be represented by a rotation matrix $$R \in \mathrm{SO}(3) = \{R\in \mathbb{R}^{3 \times 3} \mid R^{\transpose}R = I_3, \det(R) = 1\}. $$
The Lie algebra element becomes the angular velocity $\omega$ in body-fixed frame. The reconstruction equation can be expressed as
\begin{equation}
    \dot{R} = R\omega^{\wedge}. 
\end{equation}
We can write the \emph{forced} Euler-Poincar\'{e} equations~\cite{Bloch1996} as
\begin{align}
\label{eq:forcedEP}
\begin{aligned}
    J_b \dot{\omega}  &= \omega \times J_b \omega + u. 
\end{aligned}
\end{align}
where $J_b$ is the inertia matrix in body frame and $u$ is the torque applied in the body fixed principle axes. 

Thus, the error between two configuration becomes $$\Psi = R_d^{\transpose}R=:\exp{(\psi}),$$ and the corresponding angular velocity error becomes $$\dot{e} := \omega - R_d^{\transpose}R\omega_d.$$ 
\subsection{PD tracking controller}
We design the tracking controller for a system on SO(3). Referring to \cite{bullo1999tracking} and \cite{lee2011geometric}, the PD controller can be designed as a sum of feedback $F_{PD}$ and feed-forward $F_{ff}$ as
\begin{equation}
\begin{aligned}
    u &= F_{PD} + F_{ff}, \\
    F_{PD} &= -K_p(\Psi^{-1}\nabla h(\Psi))^{\vee} - K_d\dot{e}, \\
    F_{ff} &= \omega \times J_b\omega - J_b(\omega^{\wedge} R^{\transpose}R_d\omega_d - R^{\transpose}R_d\dot{\omega}_d),
\end{aligned}
\end{equation}
where $K_p$ and $K_d$ are gains for the error and velocity error, respectively. Based on the Lyapunov function $h$ we proposed in the last section, we can design the feedback term as
\begin{equation}
    F_{PD, \textnormal{proposed}} = -K_p\psi - K_d\dot{e}. 
\end{equation}

\subsection{Trajectory optimization by iterative LQR}
We consider an unconstrained trajectory optimization problem on Lie group of the following form.
\begin{problem}
\label{prob:nmpc}
Find $u_t$ such that
\begin{align*}
    \min_{u_t} \quad & N(X_{t_f}, \xi_{t_f}) + \int_{0}^{t_f} L(X_t, \xi_t, u_t) \ dt \\
    \text{s.t. } & \dot{X}_t =X \xi_{t}^{\wedge} \\
    & \dot{\xi}_t  = f(\xi_t, u_t) \\
    & \xi_0 = \xi(0) , X_0 = X(0),
\end{align*}
\end{problem}
\noindent where $t_f$ is the final time, $N(\cdot)$ is the terminal cost, $L(\cdot)$ is the stage cost.

To solve Problem \ref{prob:nmpc}, we adopt the iLQR framework \cite{ct-LQR-LieGroup} that approaches the (local) optimum by iteratively solving the LQR sub-problem (backward pass) and rolls out a new trajectory based on the optimal policy by the LQR sub-problem (forward pass). In the backward pass, a LQR problem is formulated by the perturbed equation of motion and cost function around the trajectory from the last iteration. Then an optimal linear feedback policy is obtained. In the forward pass, the optimal policy is applied to the nominal equation of motion to integrate a new trajectory. 

Research in  \cite{ct-LQR-LieGroup} derived the continuous time LQR sub-problem in the Lie algebra via calculus in a Banach space. For simplification, we use Taylor series to obtain the result. Consider $\{X_t, \xi_t, u_{t}\}$ as the trajectory and $\{X_{d,t}, \xi_{d,t}, u_{d,t}\}$ as the trajectory from the last iteration. Then the dynamics of the perturbed state $\Psi_t$ can be obtained via \eqref{eq:error_dynamics}. Given the first-order approximation of the exponential map
$$\Psi_t=\exp(\psi_t) \approx I + {\psi}^{\wedge}_t, $$
and a first-order approximation of the adjoint map $$\mathrm{Ad}_{\Psi_t} \approx \mathrm{Ad}_{I + {\psi_t}^\wedge},$$
we can linearize~\eqref{eq:error_dynamics} by discarding the second-order terms as
\begin{equation}
    \frac{d}{dt} \Psi_t \approx \frac{d}{dt}(I + {\psi}^{\wedge}_t) \approx (I + {\psi}^{\wedge}_t)({\xi}_t - \mathrm{Ad}_{(I - {\psi}_t^{\wedge})}{\xi}_{d,t})^{\wedge},
\end{equation}
\begin{equation}
\label{eq:se3_error_lin}
    \begin{aligned}
    \dot{\psi}_t = -\mathrm{ad}_{\xi_{d,t}}\psi_t + \xi_t - \xi_{d,t}.
    \end{aligned}
\end{equation}
Equation~\eqref{eq:se3_error_lin} is the linearized perturbed state in the Lie algebra. Then we define the perturbed twists $\delta \xi_t$ and perturbed inputs $\delta u_t$ as:
\begin{equation}
\delta \xi_t = \xi_t - \xi_{d,t}, \quad \delta u_t = u_t - u_{d,t}. 
\end{equation}
The perturbed twist dynamics becomes:
\begin{equation}
    \delta\dot{\xi}_t = F_t\delta \xi_t + G_t \delta u_t.
\end{equation}
where $F_t$ and $G_t$ are Jacobians of $f(\xi_t, u_t)$ around trajectory about $\xi_t$ and $u_t$. 
We define the perturbed state as $$x_t:=\begin{bmatrix} \psi_t \\ \delta \xi_t \end{bmatrix}, v_t = \delta u_t $$ and we then have the linearized perturbed state trajectory
\begin{equation}
\label{eq:pertuebed_system}
    \begin{aligned}
    \dot{x}_t &= A_t x_t + B_t v_t.  \\
    A_t &= \begin{bmatrix}
    -\mathrm{ad}_{\xi_{d,t}}\psi_t & I \\
     0 & F_t 
    \end{bmatrix}, B= \begin{bmatrix}
    0 \\
     G_t
    \end{bmatrix} 
    \end{aligned}
\end{equation}

We design the stage cost $q$ and terminal cost $p$ as
\begin{equation}
    \begin{aligned}
    q(x,v) &= \frac{1}{2}(x - x_d)^{\transpose}Q(x - x_d) + \frac{1}{2}(v - v_d)^{\transpose}S(v - v_d), \\
    p(x) &= \frac{1}{2}(x - x_d)^{\transpose}V(x - x_d).
    \end{aligned}
\end{equation}
The cost matrix $Q, V$ and $S$ are set by the user and remains constant during all iterations. The desired state $x_d$ and $v_d$ are updated in each iteration. 
Based on the perturbed state trajectory and the local cost function, we can derive the local LQR problem as follows.
\begin{problem}
\label{prob:cmpc}
Find feedforward $v_{ff,t}$ and linear feedback $K_t$ such that
\begin{align*}
    \min_{K_t, v_{ff,t}} \quad & p(x_f) + \int_{0}^{t_f} q(x_t, v_t) \ dt \\
    \text{s.t. } & \dot{x}_t = A_tx_t + B_t(v_{ff,t} + K_tx_t) \\
    & x_0 = 0.
\end{align*}
\end{problem}
Then we discretize the perturbed system \eqref{eq:pertuebed_system} and solve a discrete LQR in Problem \ref{prob:dtcmpc}.
\begin{problem}
\label{prob:dtcmpc}
Find the feed-forward $v_{ff, n}$ and feedback gain $K_n$ such that
\begin{align*}
    \min_{v_{ff,n}, K_n} \quad & p(x_N) + \sum_{n=1}^{N-1}q(x_n, v_n) \\
    \text{s.t. } & x_{n+1} = A_{n}x_n + B_n(v_{ff,n} + K_nx_n) \\
    & x_0 = 0
\end{align*}
\end{problem}
Suppose the sampling time step is $\Delta t$ and denote the time stamp by $n$, the matrix $A_n$ and $B_n$ in Problem~\ref{prob:dtcmpc} can be obtained by a zero-order hold. 

Now, we can obtain the solution to this LQR problem via dynamic programming \cite{control_limited_ddp} that solve a one step optimal control problem in each backward step.
\begin{problem}
\label{prob:cmpc_onestep}
Given the optimal cost-to-go $p_{n+1}$ at time step $n+1$, find the feed-forward $v_{ff, n}$ and feedback gain $K_n$ such that
\begin{align*}
    \min_{v_{ff,n}, K_n} \quad & p_{n+1}(x_{n+1}) + q(x_n, v_n) \\
    \text{s.t. } & x_{n+1} = A_{n}x_n + B_n(v_{ff,n} + K_nx_n) \\
\end{align*}
\end{problem}
Let the subscripts of $q$ and $p$ denotes the derivative and Hessian. The main process to obtain the optimal control policy is:
\begin{equation}
\label{eq:dp_1}
\begin{aligned}
    q_{x,n} &= -Qx_{d,n} + A_n^{\transpose}p_{x,n+1} \\
    q_{u,n} &= -Su_{d,n} + B_n^{\transpose}p_{x,n+1} \\
    q_{xx,n} &= Q + A_n^{\transpose}p_{xx,n+1}A_n \\
    q_{ux,n} &= B_n^{\transpose}p_{xx,n+1}A_n \\
    q_{uu,n} &= S + B_n^{\transpose}p_{xx,n+1}B_n \\
    v_{ff,n} &= -q_{uu,n}^{-1}q_{u,n}, K_n=-q_{uu,n}^{-1}q_{ux,n}
\end{aligned}
\end{equation}
The cost-to-go at each iteration are updated by:
\begin{equation}
\label{eq:dp_2}
\begin{aligned}
    p_{x,N} &= -Vx_{d,N},\ \ \ p_{xx,N} = V \\
    p_{x,n} &= q_{x,n+1} - q_{ux,n+1}^{\transpose}q_{uu,n+1}^{-1}q_{u,n+1} \\
    p_{xx,n} &= q_{xx,n+1} - q_{ux,n+1}^{\transpose}q_{uu,n+1}^{-1}q_{ux,n+1} \\
\end{aligned}
\end{equation}
In the forward pass, we denote the new trajectory with $\hat{(\cdot)}$ and roll out the trajectory by:
\begin{equation}
\label{eq:ilqr_roll}
\begin{aligned}
    \hat{X}_{d, 0} &= \hat{X}_{d, 0}, \ \hat{\xi}_{d, 0} = \hat{\xi}_{d, 0}. \\
    \hat{u}_{d,n} &= u_{d,n} +\gamma_n v_{ff,n} + K_n \begin{bmatrix} \log{({X}_{d, n}^{-1}\hat{X}_{d,n})}\\
    \hat{\xi}_{d,n} - \xi_{d,n} 
    \end{bmatrix} \\
    \hat{X}_{d,n+1} &= \hat{X}_{d,n}\exp{(\hat{\xi}_n \Delta t)}, \hat{\xi}_{d,n+1} = \hat{\xi}_{d,n} + \Delta t f(\hat{\xi}_n, \hat{u}_n),
\end{aligned}
\end{equation}
where $\gamma_n$ is the line search step length and $f$ is obtained in \eqref{eq:forcedEP} for the SO(3) case. In our implementation, we set $\gamma_n = 1$ for simplification. The main process of the proposed iLQR is concluded in Algorithm \ref{al:iLQR}.

\begin{algorithm}[t]
\caption{Iterative LQR on Lie Group}
\label{al:iLQR}
\small 
\begin{algorithmic}[1]
\State \textbf{Input:} Cost matrix $Q, S, V$, goal state $\{X_{g,k}, \xi_{g,k}, u_{g,k}\}$, iteration number $N$. 
\State \textbf{Initialize:} Initial trajectory: $\mathcal{X}:=\{X_{d,k}, \xi_{d.k}, u_{d,k}\}$.
\State \hspace{1.3cm} Set desired state in each iteration as: 
$$
x_{d,k} = \begin{bmatrix} \log{(X_{d,k}^{-1}X_{g,k})}\\ \xi_{g,k} - \xi_{d,k} \end{bmatrix}, v_{d,k} = u_{g,k} - u_{d,k}. 
$$

\For{$i \in (0,\dots,N)$}
    \State Obtain the perturbed system \eqref{eq:pertuebed_system} around $\mathcal{X}$. 
    \State Solve Problem~\ref{prob:dtcmpc} by \eqref{eq:dp_1} and \eqref{eq:dp_2} to obtain the optimal policy $\mathcal{U}:=\{K_{n}, v_{ff,n}\}$. \Comment{Backward pass}
    \State Roll out the new trajectory $\hat{\mathcal{X}}$ by \eqref{eq:ilqr_roll} given $\mathcal{U}$. \Comment{Forward pass}
    \State $\mathcal{X} \gets \hat{\mathcal{X}}$
    \State Update the desired state $$
x_{d,k} = \begin{bmatrix} \log{(X_{d,k}^{-1}X_{g,k})}\\ \xi_{g,k} - \xi_{d,k} \end{bmatrix}, v_{d,k} = u_{g,k} - u_{d,k}. 
$$
\EndFor 
\State \textbf{return} $\mathcal{X}, \mathcal{U}$
\end{algorithmic}
\end{algorithm}


\section{Numerical simulation}
\label{sec:sim}

In this section, we provide the simulation of the proposed PD controller and the iLQR. 
\subsection{PD control on SO(3)}
We here compare the proposed controller with the baseline provided in paper \cite{lee2010geometric} and \cite{bullo1999tracking}, where the an error function and its derivative are explicitly designed:
\begin{equation}
    \label{eq:baseline_error}
    \begin{aligned}
    h&=\frac{1}{2}\mathrm{tr}(I - \Psi), \\
    \Psi^{-1}\nabla h&=\frac{1}{2}(\Psi - \Psi^{-1}). 
    \end{aligned}
\end{equation}
Thus, the baseline controller can be expressed as: 
\begin{equation}
    F_{PD, \textnormal{baseline}} = -\frac{1}{2}K_p(\Psi - \Psi^{-1})^{\vee} - K_d\dot{e}. 
\end{equation}

\begin{table}
\caption{PD Control Parameters.}
\centering
\label{table:pd_sim_param}
\small 
\begin{tabular}{cc|cc}
\hline
{$K_p$} & { $\text{diag}(1000, 1000, 1000)$ }                  & {$\omega_{d,x}$ }   & {$\sin(0.2t+0.1)$}\\
{$K_d$} & { $\text{diag}(100, 100, 100)$ }               & {$\omega_{d,y}$ }   & {$\sin(0.3t+\frac{\pi}{5})$}\\
{$J_b$} & { $\text{diag}(1, 3, 5)$ } & {$\omega_{d,z}$ }   & {$\sin(0.1t+\frac{\sqrt{2}}{3})$}\\
{$R_0$}      & {$0.999\pi$}                           & {$\omega_{0}$ } & { $[0,0,0]$}\\
\hline
\end{tabular}
\end{table}

We define a time-varying trajectory by manually setting the body-fixed frame angular velocity as sinusoidal waves. We tested the case with a large initial error that is approaching $\pi$. The simulation parameters are presented in Table \ref{table:pd_sim_param}. We can see that the error still converges fast using the proposed controller. However, the response of the baseline controller is much slower at the initial pose. The tracking performance is presented in Fig. \ref{fig:sim_PD_tracking}.

By Rodrigues' rotation formula
$$\Psi = \exp{(\psi)} = I + \frac{\sin{\|\psi\|}}{\|\psi\|}\psi^{\wedge} + \frac{1 - \cos{\|\psi\|}}{\|\psi\|^2}\psi^{\wedge 2},$$
we can verify that the proportional term of the baseline feedback is:
$$(\Psi^{-1}\nabla h)^{\vee}= \frac{\sin{\|\psi\|}}{\|\psi\|}\psi. $$
This suggests that when $\|\psi\|$ approaches $\pi$, the proportional feedback approaches 0. This effect can be explained by the fact that the error function \eqref{eq:baseline_error} is bounded with respect to $\psi$. Thus, using \eqref{eq:baseline_error} as the error function makes the gradient vanish when the error gets its maximum value at $\|\psi\|=\pi$. However, the gradient of the proposed error function does not vanish thus enables faster convergence even when $\|\psi\|$ is near $\pi$. This effect has been illustrated in Fig. \ref{fig:cover}. For a rigorous proof of exponential stability of the second-order system, we can follow the process in \cite{bullo1999tracking} and incorporate the Lyapunov function in Theorem \ref{theorem:Lyap_candiadate}.  

\begin{figure}[t]
    \centering
    \includegraphics[width=1\columnwidth]{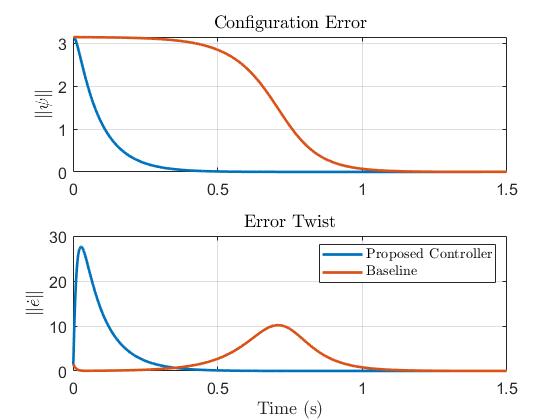}
    \caption{Tracking a time-varying trajectory by PD controller. The baseline controller responses slower when the initial pose error is approaching $\pi$. The proposed controller still respond fast despite the initial error. }
    \label{fig:sim_PD_tracking}
\end{figure}


\subsection{iLQR on SO(3)}

\begin{table}[t]
\caption{iLQR Simulation Parameters.}
\centering
\label{table:sim_param}
\small 
\begin{tabular}{cc|cc}
\hline
{$t_f$} & { 3 sec }                    & {$R_g$ }          & {$0.995\pi$}\\
{$\Delta t$} & { 0.01 sec }            & {$R_{d,k}$ }      & { $I$}\\
{$J_b$} & { $\text{diag}(5, 10, 15)$ } & {$\omega_{g}$ }   & { $[0,0,0]$}\\
{$u_{g,k}$}      & {$[0,0,0]$}                           & {$\omega_{d,k}$ } & { $[0,0,0]$}\\
\hline
\end{tabular}
\end{table}

\begin{table}[t]
\caption{iLQR Parameters.}
\centering
\label{table:control_param}
\small 
\begin{tabular}{c|ccc}
\hline
{}    & { iLQR }           & {DDP }            & {Improved DDP}\\\hline
{$V$} & { $1000I$ }         & {Adaptive}        & { $1000I$ }\\
{$Q$} & { $0$ }            & {$0$ }            & { $0$}\\
{$S$} & { $0.01I$}         & {Adaptive } & { $0.01I$}\\
{$V_g$} & { N/A}         & {$1000I$ } & { N/A}\\
\hline
\end{tabular}
\end{table}

We compare the proposed framework with the open source discrete time Lie group DDP algorithm in \cite{lie-ddp}. This baseline algorithm designs the cost function on manifold thus the gradient and Hessian matrix need to be updated at each iteration. As the second-order derivative of the dynamics is incorporated, the DDP algorithm can reach super linear convergence around the local optimum. As the proposed iLQR algorithm omitted the second order derivative of the discrete dynamics, only linear convergence is possible. For this reason, iLQR can be considered as a simplification of DDP. We also modified the cost function of DDP to ours to show the difference, which we call DDP*. 

We consider optimizing a trajectory that rotates a rigid body from the identity $I$ to a randomly generated pose $R_{g}$ that is $0.995 \pi$ far from $I$. The quaternion representation of $R_g$ is $[0.0157,0.5627,0.2839,-0.7762]$. The system is initialized with all states at the origin and input 0. All the simulation parameters are listed in Table~\ref{table:sim_param}. 

For the cost function design, we only penalize the terminal state and inputs in the stage cost. Thus, for iLQR and DDP* the terminal cost becomes: 
$$\frac{1}{2}\|(\psi - \psi^*)\|_{V_{\psi}}^{2} + \frac{1}{2}\|\omega\|_{V_{\omega}}^{2}$$
where $\exp{(\psi^*)} = R_d^{-1}R_g$. The subscript in $V_{(\cdot)}$ is to denote the block of $V$ corresponds to $\psi$ or $\omega$. At each iteration we compute the $\psi^*$ based on the terminal configuration of trajectory. The DDP in \cite{lie-ddp}, applied the cost function 
$$\frac{1}{2}\mathrm{tr}((I - R_d^{-1}R)^{\transpose}V_{g, \psi}(I - R_d^{-1}R)) + \frac{1}{2}\|\omega\|_{V_{g,\omega}}^{2}$$ 
to indicate the configuration error. The gradient and Hessian need to be updated in each iteration. The detail of the control parameters are provided in Table~\ref{table:control_param}. 



We use the norm of the difference between the final input and input in each iteration to indicate the convergence rate. The convergence is shown in Fig. \ref{fig:sim_DDP_converge}. We can see that the original DDP converges extremely slowly when the system is far from the optimum. The iLQR exhibits linear convergence rate after a few iterations. The DDP* converges in 7 iterations. Note that the iLQR converges faster in the first 5 iterations. If we combine the iLQR and DDP* as in \cite{lie-ddp}, it could take fewer iteration to converge. 


\begin{figure}[t]
    \centering
    \includegraphics[width=1\columnwidth]{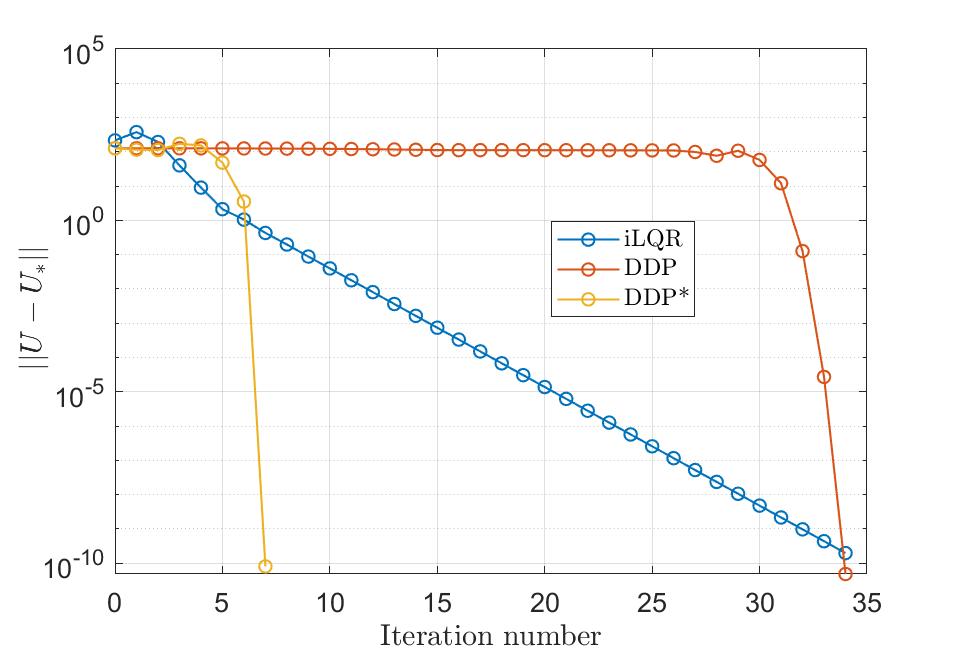}
    \caption{The convergence of different algorithms. We use $U_*$ to denote the final solution. The iLQR exhibits linear convergence rate. The original DDP converges extremely slow at the first 30 iterations when it is far from the local optimum. When equipped with the proposed cost function, the DDP* converges in 7 iterations.}
    \label{fig:sim_DDP_converge}
\end{figure}

\begin{figure}[t]
    \centering
    \includegraphics[width=1\columnwidth]{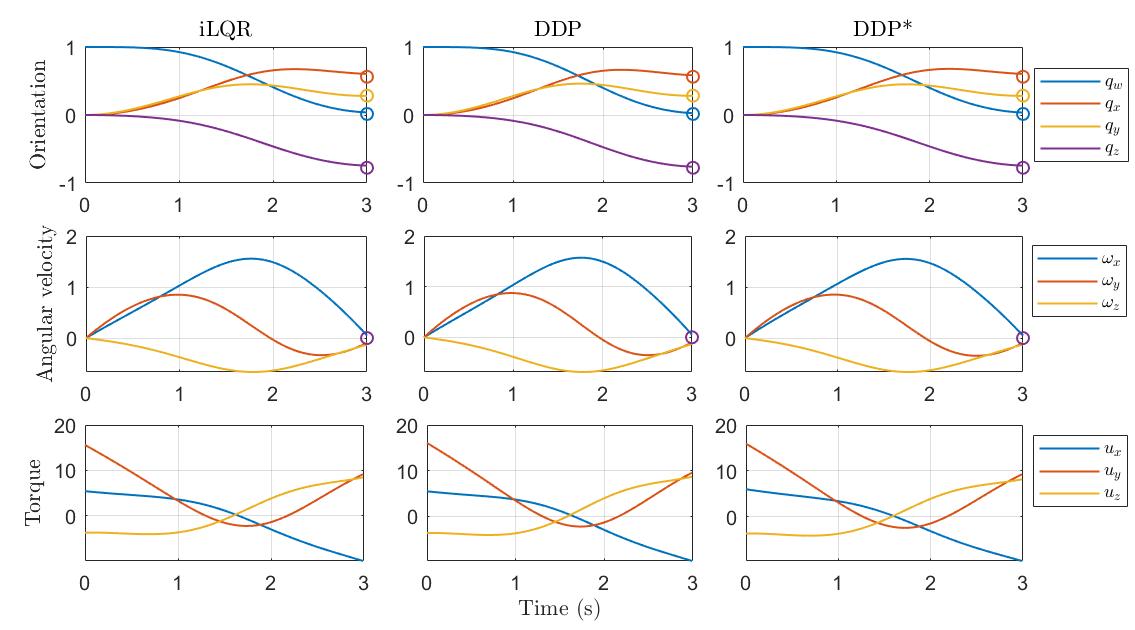} \caption{All three methods converged to a same solution. We presented the orientation indicated by the quaternion, angular velocity and torque. The circle indicates the desired terminal states.}
    \label{fig:sim_DDP_trajectory}
\end{figure}

\section{Discussion}
\label{sec:discuss}
We developed a quadratic function on the Lie algebra and derived its gradient for control on Lie groups. We show that it can be applied to design exponentially stable tracking controllers and accelerate trajectory optimization when the state space evolves on a Lie group.

Research \cite{Ilya-SO3} has shown that a continuous control law cannot globally stabilize SO(3) due to its topological properties. In this work, we note that the logarithmic map can generate discontinuous values by clamping the value in the principal branch. Thus, our PD controller can generate discontinuous control laws which ensure global convergence. 

Some estimation frameworks have fully utilized the Lie group structure, such as the Invariant EKF \cite{barrau2017invariant} and the unified DDP \cite{kobilarov2015differential} framework for perception and control. One future direction is to combine the proposed framework with perception systems such as the work of~\cite{hartley2020contact,lin2021legged}. 




\section{Conclusion}
\label{sec:conclusion}

We studied the problem of geometric control on Lie groups. This work provides a novel insight for designing a quadratic cost function in the Lie algebra via its gradient for control on Lie groups that exploit the symmetry structure of the group. We constructed this cost function via shaping the gradient and introducing a proper left-invariant metric. Based on the proposed cost function, we designed a PD controller for tracking and an iLQR for trajectory optimization. The proposed cost function enables global exponential convergence in tracking control and greatly accelerates trajectory optimization. 

{\small 
\balance
\bibliographystyle{IEEEtran}
\bibliography{bib/strings-full,bib/ieee-full,bib/references}
}

\end{document}